\documentclass[12pt]{amsart}
\usepackage{latexsym}
\usepackage{enumitem}
\usepackage{graphicx}
\usepackage{subfig}
\usepackage{float}
\usepackage{relsize}
\usepackage[top=3.2cm,bottom=3.8cm,left=3cm,right=2cm]{geometry}
\usepackage{mathrsfs}
\usepackage{amssymb}
\usepackage{amsbsy}
\usepackage{amsmath}
\usepackage{CJK}
\usepackage{tikz}
\usepackage{bm}
\usepackage{xcolor}
\usepackage[colorlinks,linkcolor=blue,citecolor=blue,pagebackref]{hyperref}
\renewcommand {\thefootnote}{\fnsymbol{footnote}}

\def \dfrac {\displaystyle \frac}

\newtheorem{thm}{Theorem}[section]
\newtheorem{lem}[thm]{Lemma}
\newtheorem{cor}[thm]{Corollary}

\theoremstyle{definition}

\def\r{\mathbb{R}^n}
\def\s{\mathbb{S}^{n-1}}

\newcommand{\Sch}{\mathcal{S}}
\newcommand{\V}{\operatorname{vol}}
\newcommand{\dd}{\mathop{}\!\mathrm{d}}

\def\R{\mathrm{R}}

\allowdisplaybreaks[4]

\numberwithin{equation}{section}

\newcommand\nnfootnote[1]{%
  \begin{NoHyper}
    \renewcommand\thefootnote{}\footnote{#1}%
    \addtocounter{footnote}{-1}%
  \end{NoHyper}
}

\makeatletter
\SetLabelAlign{hang}{%
  #1%
  \aftergroup\adjustparshapeindent}
\newcommand*\adjustparshapeindent{%
  \@ifnextchar\egroup
    {\aftergroup\adjustparshapeindent}
    {\adjustparshapeindent@auxi}}
\newcommand*\adjustparshapeindent@auxi{%
  \unless\ifdim\wd\@tempboxa=\labelwidth
    \adjustparshapeindent@auxii
  \fi}
\newcommand*\adjustparshapeindent@auxii{%
  \dimen@ = \dimexpr\wd\@tempboxa-\labelwidth\relax
  \labelwidth = \wd\@tempboxa
  \advance\linewidth -\dimen@
  \advancemargin \dimen@
  \advance\@totalleftmargin \dimen@
  \parshape \@ne \@totalleftmargin \linewidth}
\makeatother

\begin{document}
\begin{center}
  {\large \bf An affirmative solution to the
  generalized Busemann--Petty problem with subspace dimensions $2$ and $3$}
\end{center}

\vskip 10pt

\begin{center}
  {\bf Cheng\enspace Lin\ \quad\quad \  Yu-de\enspace Liu\ \quad\quad \ Ge\enspace Xiong}\\~~ \\
  \small{School of Mathematical Sciences, Key Laboratory of Intelligent Computing and Applications (Ministry of Education), Tongji University, Shanghai, 200092, China}
\end{center}

\vskip 5pt

\nnfootnote{E-mail addresses: 1. lcbruce@foxmail.com;\ 2. liuyude@fudan.edu.cn;\  3. xiongge@tongji.edu.cn.}

\begin{center}
  \begin{minipage}{14cm}
  {\bf Abstract:}
  The generalized Busemann--Petty problem asks whether origin-symmetric convex bodies in $\r$ having larger volume of all $m$-dimensional sections necessarily have larger volume. When $m\geq 4$, this is known to be false, but the cases $m=2, 3$ for $n\geq 5$ have remained open since the 1990s. In this paper, we resolve these cases. 
Together with the known results, the
generalized Busemann--Petty problem is completely solved: the answer  is affirmative for $m=1, 2, 3$, and negative for $m\geq 4$.

  \end{minipage}
\end{center}

\begin{CJK*}{UTF8}{gbsn}
\vskip 20pt
\section{\bf Introduction}
\label{1}
\vskip 10pt

The setting of this article is the $n$-dimensional Euclidean space $\r$. Let $B_n$ and $\s$ denote the unit ball and the unit sphere in $\r$, respectively.  A \emph{convex body} in $\r$ is a compact convex set with nonempty interior. Denote by $\V_m(\cdot)$ the $m$-dimensional Lebesgue measure. Write 
$G(n,m)$ for the Grassmann manifold of $m$-dimensional linear subspaces of
$\mathbb{R}^n$. Throughout this article, assume that the integer $m\in \{1,2,\ldots,n-1\}$.  

The classical \emph{Busemann--Petty problem} \cite{BP}, which resulted
from reformulating a problem in Minkowskian geometry, asks:
 
$\mathbf{(BP).}$ If $K$ and $L$ are origin-symmetric convex bodies in $\r$, is there the implication 
   \begin{align*}
        \V_{n-1}(K\cap E)\leq \V_{n-1}(L\cap E),\quad \forall\thinspace E\in G(n,n-1)\quad\Longrightarrow\quad \V_n(K)\leq \V_n(L)?
    \end{align*}  

The Busemann--Petty problem has a long and influential history. A negative answer to the problem for $n\ge 5$ was established in a sequence of papers by Larman and Rogers \cite{LR} (for $n\geq 12$), Ball \cite{Ball} ($n\geq 10$), Giannopoulos \cite{Gi} and Bourgain \cite{Bo} $(n \geq 7)$, Papadimitrakis \cite{Pa} and Gardner \cite{Gardner}  $(n \geq 5)$.  In 1988, Lutwak \cite{Lutwak2} took the first step towards a \emph{full} solution by introducing the notion of \emph{intersection body}. It turns out that the Busemann--Petty problem is equivalent to asking whether all origin-symmetric convex bodies in $\r$ are intersection bodies (Lutwak himself proved one direction of the equivalence in \cite{Lutwak2}, and the other was shown independently by Gardner \cite{Gardner} and Zhang \cite{Zhang1}). Following this route, Gardner \cite{Gardner1} and Zhang \cite{Zhang} answered affirmatively  the Busemann--Petty problem for $n=3$ and $n=4$, respectively. A unified solution was obtained by   Gardner,  Koldobsky and  Schlumprecht \cite{GKS}, using Koldobsky's functional analysis approach \cite{Ko}.  We  refer readers to \cite[Note 8.9]{GardnerBook} and \cite[pp. 3--5]{Ko2} for
more historical comments.

Since the answer to the Busemann--Petty problem is negative for $n\geq 5$, Milman and Pajor \cite{MilmanPajor1989} posed the \emph{isomorphic Busemann--Petty problem}:
 Does there exist an absolute constant $c>0$ such that for all $n$ and all pairs of origin-symmetric convex bodies $K,L$ in $\mathbb{R}^n$, the following implication holds
   \begin{align*}
        \V_{n-1}(K\cap E)\leq \V_{n-1}(L\cap E),\quad \forall\thinspace E\in G(n,n-1)\quad\Longrightarrow\quad \V_n(K)\leq c\V_n(L)?
    \end{align*} 
 It is worth mentioning  (see \cite{MilmanPajor1989}) that the isomorphic BP problem is equivalent to the celebrated Bourgain slicing problem \cite{Bourgain1986}, which was resolved by Klartag and Lehec \cite{KlartagLehec2025}  using Guan's bound \cite{Guan2024}. Very recently, Koldobsky and Zvavitch \cite{KoldobskyZvavitch2026} proved the sharp order of the constant involved in the isomorphic BP problem for arbitrary measures.  
\vskip 3pt
In \cite{Zhang96}, Zhang posed a natural generalization of the Busemann--Petty problem,  known as the \emph{generalized Busemann--Petty problem} (see also \cite[Problem 8.2]{GardnerBook}).
\vskip 3pt
$\mathbf{(GBP).}$ If $K$ and $L$ are origin-symmetric convex bodies in $\r$, is there the implication 
    \begin{align*}
        \V_m(K\cap E)\leq \V_m(L\cap E),\quad \forall\thinspace E\in G(n,m)\quad\Longrightarrow\quad \V_n(K)\leq \V_n(L)?
    \end{align*}
    
   It is obvious that if $m=1$, the GBP problem is trivially true; if $m=n-1$, the GBP problem is indeed the classical BP problem.  In  \cite{Zhang96}, Zhang  introduced the notion of  \emph{$m$-intersection body} and proved that the GBP  problem is equivalent to asking
whether all origin-symmetric convex bodies are  $m$-intersection bodies. 
In 1999,  it was shown by Bourgain and Zhang \cite{BZ} (see also a correction by Rubin and Zhang \cite{RZ}), and later by Koldobsky 
\cite{Ko1}, that the answer to the GBP problem is negative for $m\geq 4$. Please refer to, e.g., \cite{BZ}, \cite{GZ},  \cite{Mi1}, \cite{Mi}, \cite{Mi2}, \cite{RZ}, \cite{Ru} and \cite{Zhang96} for partial solutions and work relevant to the GBP  problem.  
 Until now,  the GBP problem for $m=2, 3$ remained open. 
 \vskip 3pt
 In this article, the long-standing GBP problem is definitively resolved.

\begin{thm}
\label{1.1}
Let $K$ and $L$ be origin-symmetric convex bodies in $\r$. If $n\geq 4$ and $m=2, 3$, then
    \begin{align*}
        \V_m(K\cap E)\leq \V_m(L\cap E),\quad \forall\thinspace E\in G(n,m)\quad\Longrightarrow \quad\V_n(K)\leq \V_n(L).
    \end{align*}
\end{thm}
 Consequently, together with the known results, the GBP problem is \emph{completely} solved: the answer  is affirmative for $m=1, 2, 3$, and negative for $m\geq 4$. Additionally, for $n=4$ and $m=3$, Theorem \ref{1.1} reduces to the classical BP problem in $\mathbb{R}^4$ solved by Zhang \cite{Zhang}.

In the following, we outline the  idea underlying the proof of Theorem \ref{1.1}. 

First, we observe that to prove Theorem \ref{1.1}, it suffices to prove the case
$m=3$. Indeed,  if $\V_2(K\cap F)\leq \V_2(L\cap F)$ for all $ F\in G(n,2)$, then for all $E\in G(n,3)$, $\V_2(K\cap E\cap F)\leq \V_2(L\cap E\cap F)$.
 By the affirmative solution to the classical BP problem in $\mathbb{R}^3$, we have  $\V_3(K\cap E)\leq \V_3(L\cap E)$ for all $E\in G(n,3)$, and therefore  $\V_{n}(K)\leq \V_{n}(L)$ by our \emph{supposed} affirmative solution.

Second, to prove the case $m=3$, it suffices to prove that each origin-symmetric convex body $K$ in $\r$ is a $3$-intersection body, which is equivalent to  constructing a \emph{positive} Borel measure $\mu_K$ on $G(n,3)$ such that  $\mathrm{R}_3^t\mu_K=\rho_K^{\,n-3}\dd\sigma,$
where $\mathrm{R}_3^t$ is the \emph{dual Funk--Radon transform} on $G(n,3)$, $\rho_K$ is the \emph{radial function} of $K$ and $\dd\sigma$ is the Haar probability measure on $\s$.  Please refer to Section \ref{2} for their definitions.

To construct such a measure $\mu_K$, it is crucial to make full use of the \emph{convexity} of \(K\) to guarantee the \emph{positivity}  of $\mu_K$. To this end,  we investigate the \emph{$3$-plane transform} $P_{f,E}(y)=\int_E f(x+y)\dd x,
\thinspace y\in E^\perp$,
and originally introduce the associated \emph{weight}
\[
w_f(E)=\det(-D^2P_{f,E}(0))
\]
for every Schwartz function $f$ and every subspace $E\in G(n,3)$. If in addition $f$ is  nonnegative, even
and log-concave (Specifically,  $f=\mathbf{1}_K$, i.e., the indicator function of $K$),  it guarantees that  $w_f(E)\geq 0$.  After reformulating $w_f(E)$ as the  identity (Lemma \ref{continuity-density}) 
\begin{align*}
    w_f(E)
=(2\pi)^{-(n-3)^2}
\det(
\int_{E^\perp}
\zeta\otimes\zeta\,\widehat f(\zeta)\dd\zeta),\quad E\in G(n,3),
\end{align*}
we establish the following \emph{vital} identity (Theorem \ref{determinantal-Crofton})
\begin{equation*}
(\mathrm{R}_3^t w_f)(u)
=(2\pi)^{n-3}
(
\int_{\mathbb{R}}f(tu)\dd t
)^{n-3},
\quad u\in\s,
\end{equation*}
which is  exactly  the ``engine" of this article.

Finally, by the Gaussian regularization $f_{\varepsilon}=\mathbf{1}_K*\gamma_\varepsilon$, i.e., the convolution of $\mathbf{1}_K$ and $\gamma_\varepsilon(x)
  =(2\pi\varepsilon^2)^{-n/2}
  e^{-|x|^2/(2\varepsilon^2)}$ for arbitrary $\varepsilon>0$, the desired measure $\mu_K$ is obtained via the $\text{weak--}^{*}$ limit of $w_{f_\varepsilon}(E)\dd E$ as $\varepsilon\to 0^+$ (see Theorem \ref{positive-representation}).
  \vskip 3pt

The article is organized as follows. In Section $\ref{2}$,  we collect some necessary facts on the Funk--Radon transform and  the $m$-plane transform. Several useful theorems, including the Pr\'ekopa theorem, the Blaschke--Petkantschin formula and the Andr\'eief identity,  are also provided. The proof of Theorem \ref{1.1} is presented in Section $\ref{3}$.

\vskip 3pt
\textbf{Declaration on the use of AI and Acknowledgment}: 
Theorem \ref{Cauchy-Binet or Andreief identity} (the Andr\'eief identity) was found with the help of ChatGPT 5.6 Sol,  which led us to prove \mbox{Theorem \ref{determinantal-Crofton}.}

The authors would like to thank Dr. Kaiwen Yang for extensive discussions and careful examination of earlier drafts.

\vskip 20pt
\section{\bf Preliminaries}
\label{2}
\vskip 5pt

\subsection{The Funk--Radon transform}
\label{2.1}
\
\vskip 5pt

As usual, write $C(\s)$ and $C(G(n,m))$ for the spaces of continuous functions on 
$\s$ and $G(n,m)$, respectively. 
 $G(n,m)$ is identified with $ {\rm O}(n)/({\rm O}(m)\times {\rm O}(n-m))$
and equipped with the quotient topology, where ${\rm O}(n)$ is the orthogonal group on $\mathbb{R}^n$. Hence,  a function
$g$ is \emph{continuous} on $G(n,m)$ if and only if $U\mapsto g(UE)$
is continuous on ${\rm O}(n)$ for a fixed $E\in G(n,m)$. See, e.g.,
 \cite[Appendix A]{RubinG} for more explanations.
 
 Write $\dd\sigma$ and $\dd E$ for the Haar probability measures on $\s$ and $G(n,m)$, respectively. The classical \emph{Funk--Radon transform} $\mathrm{R}_{m}: C(\s) \to C(G(n,m))$
is
\[
(\mathrm{R}_{m} f)(E)
=\int_{\s\cap E}f(u)\dd\sigma_E(u),\quad E\in G(n,m),\enspace f\in C(\s),
\]
where $\dd\sigma_E$ is the Haar
probability measure on $\s\cap E$. 

For an origin-symmetric convex body
$K$ in $ \r$, its \emph{radial function} is
\[
\rho_{K}(u)=\max \{\lambda \geq 0: \lambda u \in K\}, \quad u \in \s .
\]
The Funk--Radon transform $\R_m$
is closely connected with the central sections of
origin-symmetric convex bodies $K$ by the following formula
\begin{equation*}
(\mathrm{R}_{m} \rho_{K}^{m})(E)
=\frac{1}{\omega_{m}}\V_m(K \cap E),\quad E\in G(n,m),
\end{equation*}
where $\omega_m=\pi^{\frac{m}{2}}/\Gamma(1+\frac{m}{2})$ is the volume of the unit ball $B_m$ in $\mathbb{R}^m$.

Write $\mathcal{M}(\s)$ and $\mathcal{M}(G(n,m))$ for the spaces of signed Borel
measures on $\s$ and $G(n,m)$, respectively. The \emph{dual transform} $\R^t_m$ of $\R_m$ is the map $\mathcal{M}(G(n,m))\to \mathcal{M}(\s)$ given by
\[
  \int_{\s}f\dd(\mathrm{R}_m^t\mu)
  =\int_{G(n,m)}(\mathrm{R}_mf)(E)\dd\mu(E),
  \quad  f\in C(\s).
\]
If the measure $\mu$ has a continuous density $g$, then the transform may be explicitly
written in terms of $g$ (see \cite{Zhang96}) as 
\begin{align}
\label{dual}
    (\mathrm{R}_{m}^{t}g)(u)
=\int_{u \in E \in G(n,m)} g(E)\dd\nu_m(E),\quad u\in\s,
\end{align}
where $\nu_m$ is the Haar probability measure on $G(n-1, m-1)$. 

Following \cite{Zhang96}, an origin-symmetric convex body $K$ in $\r$ is called an \emph{$m$-intersection body}, if there exists a positive Borel measure $\mu_K$ on $G(n,m)$ such that $\mathrm{R}_m^t\mu_K=\rho_K^{\,n-m}\dd\sigma.$
 Zhang \cite[Theorem 7]{Zhang96} proved the following theorem.
\begin{thm}[{\cite[Theorem 7]{Zhang96}}]
\label{m-interesection}
     The GBP problem for $m\in\{1,2,\ldots,n-1\}$ has an affirmative answer in $\r$ if and only if each origin-symmetric convex body $K$ in $\r$ is an $m$-intersection body.
\end{thm}

\subsection{The \texorpdfstring{$m$}{m}-plane transform and several useful results}
\label{2.2}
\
\vskip 5pt

Let $\Sch(\mathbb{R}^n)$ be the Schwartz space on $\r$. For $f\in\Sch(\r)$, its \emph{Fourier transform} is 
\[
  \widehat f(\zeta)
  =\int_{\mathbb{R}^n}f(x)e^{-i\langle x,\zeta\rangle}\dd x, \quad \zeta\in\r,
\]
where $\langle x,\zeta\rangle$ is the standard inner product of $x$ and $\zeta$ in $\r$.

 The  Fourier inversion theorem reads:
\begin{align*}
    f(x)=(2\pi)^{-n}\int_{\mathbb{R}^n}
  \widehat f(\zeta)e^{i\langle x,\zeta\rangle}\dd\zeta,\quad x\in\r.
\end{align*}

For $E\in G(n,m)$ and  $f\in\Sch(\mathbb{R}^n)$, the \emph{$m$-plane transform}  $P_{f,E}: E^{\perp}\to \mathbb{C}$ is 
\begin{align*}
    P_{f,E}(y)=\int_E f(x+y)\dd x,
  \quad y\in E^\perp,
\end{align*}
where $E^{\perp}$ is the orthogonal complement of $E$.
Let $K$ be a convex body in $\mathbb{R}^n$. Write 
\begin{align*}
    P_{K,E}(y)
  =\int_E \mathbf{1}_K(x+y)\dd x
  =\V_m(K\cap(E+y)),
  \quad y\in E^\perp.
\end{align*}

 The slice-projection theorem \cite[Theorem 3.27]{Markoe} reads:
\begin{align}
\label{Sp}
    \widehat{P_{f,E}}(\zeta)=\widehat f(\zeta),\quad \zeta\in E^{\perp},
\end{align}
where $\widehat{P_{f,E}}$ is the Fourier transform of $P_{f,E}$ in $E^{\perp}$. Therefore, $P_{f,E}\in \Sch(E^{\perp})$. 

From the slice-projection theorem,  the Fourier inversion theorem and the definition of $P_{f,\mathbb{R}u}$, we have
\begin{align*}
    \int_{u^\perp}\widehat f(\zeta)\dd\zeta=\int_{u^\perp}\widehat{P_{f,\mathbb{R}u}}(\zeta)\dd\zeta=(2\pi)^{n-1}P_{f,\mathbb{R}u}(0)=(2\pi)^{n-1}\int_{\mathbb{R}}f(tu)\dd t,\quad u\in\s.
\end{align*}
Thus,
\begin{align}
\label{Fourier}
    \int_{u^\perp}\widehat f(\zeta)\dd\zeta=(2\pi)^{n-1}\int_{\mathbb{R}}f(tu)\dd t,\quad u\in\s.
\end{align}

 The following Theorem \ref{log-concave} is established by Pr\'ekopa in \cite[Theorems 6 and 7]{Prekopa}. 
\begin{thm}[{\cite[Theorems 6 and 7]{Prekopa}}] 

\label{log-concave}
The following assertions hold.
\begin{itemize}
    \item [$(1).$] If $f(x,y): \mathbb{R}^k\times \mathbb{R}^l \to \mathbb{R}$ is log-concave  and $A$ is a convex subset of $\mathbb{R}^l$, then $x\mapsto\int_A f(x,y)\dd y$ is log-concave in $\mathbb{R}^k$.
    \item [$(2).$] If $f$ and $g$ are log-concave functions defined on $\r$, then 
the convolution 
\begin{align*}
    f*g(x)=\int_{\r}f(x-y)g(y)\dd y,\quad x\in\r,
\end{align*}
is also log-concave.
\end{itemize}
    
\end{thm}

For $d\in\{1,2,\ldots,n\}$ and  $x_1,\ldots,x_d$ in $\r$, define 
\[
  \Delta_d(x_1,\ldots,x_d)
  =\det((\langle x_k,x_l\rangle)_{k,l=1}^d)^{1/2} 
  \quad\text{and}\quad
  s_{n,d}=\frac{2^d\pi^{nd/2}}{\Gamma_d(n/2)},
\]
where $\Gamma_d(\frac{n}{2})=\pi^{d(d-1)/4}\prod\limits_{j=0}^{d-1}\Gamma(\frac{n-j}{2})$.

The following Blaschke--Petkantschin formula (see \cite[Theorem 2.1]{RubinBP}) is needed. 

\begin{thm}
\label{oushi}
If $1\leq k\leq m\leq n$ and $h\in L^1((\mathbb{R}^n)^k)$, then
\begin{equation*}
\begin{aligned}
  \int_{(\mathbb{R}^n)^k}h(x_1,\ldots,x_k)\dd x_1\cdots \dd x_k=\frac{s_{n,k}}{s_{m,k}}
  \int_{G(n,m)}\int_{E^k}
  h(x_1,\ldots,x_k)\,
  \Delta_k(x_1,\ldots,x_k)^{n
  -m}
  \dd x_1\cdots \dd x_k\dd E.
\end{aligned}
\end{equation*}
\end{thm}

The following Andr\'eief identity (see \cite[Proposition 7.1]{Zygouras}),  an
integral version of the classical Cauchy--Binet formula, will be used. Refer, e.g., to \cite[Section~2.2]{Forrester2019} for its proof.

\begin{thm}[Cauchy--Binet or Andr\'eief identity]
\label{Cauchy-Binet or Andreief identity}
If $(X, \mu)$ is a measure space and $\{\phi_{j}(\cdot)\}_{j=1}^N, \{\psi_{k}(\cdot)\}_{k=1}^N$ $\subset L^{2}(X, \mu)$, then
\begin{align*}
    \operatorname{det}((\int_{X} \phi_{j}(x) \psi_{k}(x)\dd \mu(x))_{j,k=1}^N)
    =\frac{1}{N!}\int_{X^{N}} \operatorname{det}((\phi_{j}(x_{k}))_{j,k=1}^N) \operatorname{det}((\psi_{j}(x_{k}))_{j,k=1}^N) \dd\mu( x_{1}) \cdots \dd\mu( x_{N}).
\end{align*}

\end{thm}

From Theorem \ref{Cauchy-Binet or Andreief identity}, we establish the following identity immediately.

\begin{cor}
\label{Andreief lemma}
If $\phi:\mathbb{R}^n\to\mathbb C$ is measurable and satisfies $\int_{\mathbb{R}^n}|x|^2|\phi(x)|\dd x<\infty,$
then
\[
  \det(\int_{\mathbb{R}^n}x\otimes x\,\phi(x)\dd x)
  =\frac{1}{n!}\int_{(\mathbb{R}^n)^n}
  \Delta_n(x_1,\ldots,x_n)^2
  \mathsmaller{\prod}\limits_{j=1}^n\phi(x_j)\dd x_1\cdots \dd x_n.
\]
\end{cor}
\begin{proof}
Choose an orthonormal basis $e_1,\ldots,e_n$ of $\mathbb{R}^n$ and define
\[
\theta(x)=
\begin{cases}
\dfrac{\phi(x)}{|\phi(x)|}, & \phi(x)\neq 0,\\[6pt]
0, & \phi(x)=0.
\end{cases}
\]
 Then $x\mapsto \theta(x)\langle x,e_j\rangle$ and $x\mapsto\langle x,e_j\rangle$ belong to $ L^2(\mathbb{R}^n, |\phi(x)|\dd x)$, since 
 \begin{align*}
     \int_{\r}|\langle x,e_j\rangle|^2|\phi(x)|\dd x\leq \int_{\mathbb{R}^n}|x|^2|\phi(x)|\dd x<\infty.
 \end{align*}
 
 Putting $(X, \mu)=(\mathbb{R}^n, |\phi(x)|\dd x)$, $N=n$, $\phi_j(x)=\theta(x)\langle x,e_j\rangle$, and $\psi_k(x)=\langle x,e_k\rangle$ into Theorem \ref{Cauchy-Binet or Andreief identity}, we  establish the desired identity.
\end{proof}
In \cite{LYZ}, Lutwak,  Yang and  Zhang established  a similar identity  on the unit sphere $\s$ by mixed discriminants.

\vskip 20pt
\section{\bf Proof of Theorem \ref{1.1}}
\label{3}
\vskip 10pt

Recall that for a Schwartz function  $f\in\Sch(\r)$ and a subspace $E\in G(n,m)$, the $m$-plane transform 
 $P_{f,E}(y)=\int_E f(x+y)\dd x,\; y\in E^\perp$. We define the \emph{weight} $w_f$ of $f$ by 
\[
  w_f:G(n,m)\to \mathbb{C},\quad\quad w_f(E)=\det (-D^2P_{f,E}(0)),
\]
where $D^2$ is the Hessian on  $E^{\perp}$.

\begin{lem}
\label{continuity-density}
If $f\in\Sch(\mathbb{R}^n)$, then the function $ w_f$ is continuous on $G(n,m)$ and 
\begin{align*}
    w_f(E)=(2\pi)^{-(n-m)^2}\det(\int_{E^{\perp}}\zeta\otimes\zeta\,\widehat f(\zeta)\dd\zeta),\quad E\in G(n,m).
\end{align*}
\end{lem}

\begin{proof}
Since $f\in\Sch(\r)$, by the  Fourier inversion theorem and formula \eqref{Sp}, we have 
\begin{align*}
    P_{f,E}(y)=(2\pi)^{-(n-m)}\int_{E^{\perp}}\widehat{ P_{f,E}}(\zeta)e^{i\langle x,\zeta\rangle}\dd\zeta=(2\pi)^{-(n-m)}\int_{E^{\perp}}\widehat f(\zeta)e^{i\langle x,\zeta\rangle}\dd\zeta,\quad y\in E^{\perp}.
\end{align*}
Note that the dimension of $E^{\perp}$ is $n-m$. It follows that
\begin{align*}
    w_f(E)=\det (-D^2P_{f,E}(0))=(2\pi)^{-(n-m)^2}\det(\int_{E^{\perp}}\zeta\otimes\zeta\,\widehat f(\zeta)\dd\zeta).
\end{align*}

 Fix $E_0\in G(n,m)$. For all $ \thinspace U\in {\rm O}(n)$, we have  
\[
w_f(UE_0)
=
(2\pi)^{-(n-m)^2}
\det(\int_{E_0^\perp}
\zeta\otimes\zeta\,\widehat f(U\zeta)\dd\zeta).
\]

If $U_j\to U$ in ${\rm O}(n)$, then
$\widehat f(U_j\zeta)\to\widehat f(U\zeta)$ for $\zeta\in\r$ since $\widehat{f}\in\Sch(\r)$. By $\widehat{f}\in\Sch(\r)$ again, there exists $c>0$ such that
\begin{align*}
    |\zeta|^2|\widehat f(U_j\zeta)|
\leq c|\zeta|^2(1+|U_j\zeta|)^{-(n+9)}= c|\zeta|^2(1+|\zeta|)^{-(n+9)},\quad \zeta\in\r.
\end{align*}
 Since  $c\int_{E_0^{\perp}}|\zeta|^2(1+|\zeta|)^{-(n+9)}\dd\zeta<\infty$,  the Lebesgue dominated convergence theorem gives that $w_f(U_jE_0)\to w_f(UE_0)$, as $j\to\infty$. So,
$w_f$ is continuous on $G(n,m)$.
\end{proof}

\begin{thm}
\label{determinantal-Crofton}
Let $n\geq 4$. For every $f\in\mathcal S(\mathbb R^n)$, we have
\begin{align}
\label{wf}
    (\mathrm{R}_3^t w_f)(u)
=
(2\pi)^{n-3}
(\int_{\mathbb R}f(tu)\dd t)^{n-3},
\quad u\in \s.
\end{align}
If $f\in\mathcal S(\mathbb R^n)$ is  nonnegative, even
and log-concave, then $w_f(E)\geq 0 $ for all $ E\in G(n,3)$.
\end{thm}

\begin{proof}
\textbf{Step~1.} Prove  formula \eqref{wf}.

  Since $\int_{E^{\perp}}|\zeta|^2|\widehat f|\dd\zeta<\infty$ by $\widehat f\in \Sch(\r)$,  from Lemma \ref{continuity-density}  and Corollary \ref{Andreief lemma}, we have
\[
\begin{aligned}
  &\quad\enspace w_f(E)=(2\pi)^{-(n-3)^2}\text{det}(\int_{E^{\perp}}\zeta\otimes\zeta\,\widehat f(\zeta)\dd\zeta)
  \\
  &=\frac{(2\pi)^{-(n-3)^2}}{(n-3)!}
  \int_{(E^\perp)^{(n-3)}}
  \Delta_{n-3}(\zeta_1,\ldots,\zeta_{n-3})^2
  \mathsmaller{\prod}\limits_{j=1}^{n-3}\widehat f(\zeta_j)\dd\zeta_1\cdots \dd\zeta_{n-3},\enspace E\in G(n,3).
\end{aligned}
\]

Thus, from formula \eqref{dual}, Theorem \ref{oushi} (let $m=k=n-3$), the Fubini theorem and formula \eqref{Fourier}, it follows that for all $u\in\s$,
\begin{align*}
    &\quad \enspace(\mathrm{R}_3^tw_f)(u)=\int_{u\in E\in G(n,3)}w_f(E)\dd\nu_3(E)
    =\int_{F\in G(u^{\perp},n-3)}w_f(F^{\perp})\dd F 
    \\
    &=\frac{(2\pi)^{-(n-3)^2}}{(n-3)!}\int_{F\in G(u^{\perp},n-3)}
  \int_{(F)^{(n-3)}}
  \Delta_{n-3}(\zeta_1,\ldots,\zeta_{n-3})^2 
  \mathsmaller{\prod}\limits_{j=1}^{n-3}\widehat f(\zeta_j)\dd\zeta_1\cdots \dd\zeta_{n-3}\dd F 
  \\
  &=(2\pi)^{-(n-3)^2}
    \frac{s_{n-3,n-3}}{(n-3)!\,s_{n-1,n-3}}
    \int_{(u^\perp)^{n-3}}
    \mathsmaller{\prod}\limits_{j=1}^{n-3}\widehat f(\zeta_j)
    \dd\zeta_1\cdots \dd\zeta_{n-3} \\
  &=(2\pi)^{-(n-3)^2}
    \frac{s_{n-3,n-3}}{(n-3)!\,s_{n-1,n-3}}
    (\int_{u^\perp}\widehat f(\zeta)\dd\zeta)^{n-3} 
    \\
    &=(2\pi)^{2(n-3)}\frac{s_{n-3,n-3}}{(n-3)!\,s_{n-1,n-3}}(\int_{\mathbb{R}}f(tu)\dd t)^{n-3},  \quad \quad (*)
\end{align*}
where $G(u^{\perp},n-3)$ is the Grassmann manifold of $(n-3)$-dimensional subspaces of $u^{\perp}$ and $\dd F$ is the Haar probability measure on $G(u^{\perp},n-3)$. The second equality holds since for all $u\in \s$, the map $E\longmapsto E^\perp$ 
identifies
$\{E\in G(n,3):u\in E\}$ with $G(u^\perp,n-3)$ and  carries $\dd\nu_3$ to $\dd F$ by the uniqueness
of the rotational-invariant probability measure.

Using  the definition of $s_{n,m}$ and the identity $\Gamma(s+1)=s\Gamma(s),\; s>0$,  we have
\[
\begin{aligned}
\frac{s_{n-3,n-3}}
     {(n-3)!\,s_{n-1,n-3}}
=\frac{\pi^{-(n-3)}}{(n-3)!}\mathsmaller{\prod}\limits_{j=0}^{n-4}
\frac{n-3-j}{2}
=(2\pi)^{-(n-3)}.
\end{aligned}
\]
Therefore, we obtain
\begin{align*}
    (\mathrm{R}_3^tw_f)(u)
  =(2\pi)^{n-3}
  (\int_{\mathbb{R}}f(tu)\dd t)^{n-3},
  \quad u\in\s.
\end{align*}

\textbf{Step~2.} Prove that $w_f(E)\geq0$ for all $E\in G(n,3)$ under the assumption on $f$. 

Assume that $f\in\Sch(\r)$ is  nonnegative, even
and log-concave. Then $P_{f,E}\in \Sch(E^{\perp})$ is  nonnegative and even. Moreover, $P_{f,E}$ is log-concave on $E^{\perp}$ by Theorem \ref{log-concave} (1). Thus, 
\[
  P_{f,E}(0)
  =P_{f,E}(\frac{y+(-y)}{2})
  \geq P_{f,E}(y)^{1/2}P_{f,E}(-y)^{1/2}
  =P_{f,E}(y), \quad \forall \thinspace y\in E^{\perp}.
\]
That is, $P_{f,E}$ attains its  global maximum at the origin.
 
Let $y\in E^\perp$ and $\varphi_y(t)=P_{f,E}(ty),\; t\in\mathbb R.$
Then $\varphi_y$ is twice differentiable and attains its maximum at $t=0$.
So, $\varphi_y''(0)\leq 0, $ and hence $ D^2P_{f,E}(0)[y,y]=\varphi_y''(0)\leq 0$ for all $y\in E^{\perp}$.
Consequently, $D^2P_{f,E}(0)$ is negative semidefinite and  $w_f(E)\geq0$ for all $E\in G(n,3)$.
\end{proof}

\noindent{\bf{Remark 1.}} (i). The above arguments reveal  why the case \(m=3\) is \emph{distinctive}.
For general \(m\), the Andr\'eief identity applied to
\(\det(-D^2P_{f,E}(0))\) produces the factor
\(\Delta_{n-m}^{2}\), while the
Blaschke--Petkantschin formula produces the factor
\(\Delta_{n-m}^{m-1}\).
Thus, the two powers agree exactly only if \(m-1=2\), that is, \(m=3\).

(ii). From a series of equalities in $(*)$ of Step 1, we indeed have 
\begin{align*}
    \int_{F\in G(u^{\perp},n-3)}
  \int_{(F)^{(n-3)}}
  \Delta_{n-3}(\zeta_1,\ldots,\zeta_{n-3})^2 
  \mathsmaller{\prod}\limits_{j=1}^{n-3}g(\zeta_j)\dd\zeta_1\cdots \dd\zeta_{n-3}\dd F =c(n)
    (\int_{u^\perp}g(\zeta)\dd\zeta)^{n-3}, 
\end{align*}
where $g$ is a bounded Borel function on $\r$ with  compact support. Specifically, if $g=\mathbf{1}_L$ where $L$ is  an origin-symmetric
star body, then the formula  corresponds to Lemma 6.1 in \cite{GZ} of Grinberg and Zhang for $i=3$.

\begin{thm}
\label{positive-representation}
If $K$ is an origin-symmetric convex body 
in $\r$ with $n\geq 4$, then there exists a positive Borel measure
$\mu_K$ on $G(n,3)$ such that $\mathrm{R}_3^t\mu_K=\rho_K^{n-3}\dd\sigma.$ 
\end{thm}

\begin{proof}
For $\varepsilon>0$, let $\gamma_\varepsilon(x)
  =(2\pi\varepsilon^2)^{-n/2}
  e^{-|x|^2/(2\varepsilon^2)},\; x\in\r$; and $
  f_\varepsilon=\mathbf{1}_K*\gamma_\varepsilon.$
Then  $f_\varepsilon\in\Sch(\r)$ is  nonnegative and even. Since $\mathbf{1}_K$ and $\gamma_\varepsilon$ are log-concave in $\r$, $f_{\varepsilon}$ is log-concave by Theorem \ref{log-concave} (2).

Let $w_\varepsilon=w_{f_\varepsilon}$ and $\dd\mu_\varepsilon(E)
  =(4\pi)^{3-n}w_\varepsilon(E)\dd E$ for $\varepsilon>0$.
From Lemma \ref{continuity-density} and Theorem
\ref{determinantal-Crofton}, it follows that $\mu_\varepsilon$ is a positive Borel measure
and
\[
  \mathrm{R}_3^t\mu_\varepsilon
  =(\frac12\int_{\mathbb{R}}f_\varepsilon(tu)\dd t)^{n-3}\dd\sigma(u):= h_{\varepsilon}(u)\dd\sigma(u).
\]
In the following, we divide the proof into two steps.

\textbf{Step~1.} Prove that $ \sup\limits_{u\in\s}
  |h_\varepsilon(u)-\rho_K^{n-3}(u)|\to 0$, as $\varepsilon\to 0^+$.

Since $K$ is origin-symmetric, it follows that  $P_{K,\mathbb Ru}(y)=\V_1(K\cap(\mathbb Ru+y)),\; y\in u^{\perp}$, is even and  $P_{K,\mathbb Ru}(0)=2\rho_K(u),\; u\in\s$. 
By the Brunn concavity theorem, $P_{K,\mathbb Ru}$ is concave on its
support. Hence, we have $P_{K,\mathbb Ru}(0)\geq P_{K,\mathbb Ru}(y),\; \forall\; y\in u^\perp.$

Assume $0<r\leq R<\infty$ such that $rB_n\subset K\subset RB_n$. For $y\in u^{\perp}$ with
$0<|y|\leq r/2$, let $z=ry/(2|y|)$ and $\lambda=2|y|/r$. 
Since $z\in rB_n\subset K$, the point $z$ lies in the support of $P_{K,\mathbb Ru}$. By the concavity of $P_{K,\mathbb{R}u}$, it follows that
\[
\begin{aligned}
  P_{K,\mathbb Ru}(y)
  =
  P_{K,\mathbb Ru}((1-\lambda)0+\lambda z) \geq
  (1-\lambda)P_{K,\mathbb Ru}(0)
  +\lambda P_{K,\mathbb Ru}(z) \geq
  (1-\lambda)P_{K,\mathbb Ru}(0).
\end{aligned}
\]
Added with  $\lambda=2|y|/r$ and $K\subset RB_n$, we have that for $y\in u^{\perp}$ with
$0<|y|\leq r/2$,
\[
  P_{K,\mathbb Ru}(0)-P_{K,\mathbb Ru}(y)
  \leq
  \lambda P_{K,\mathbb Ru}(0)
  \leq
  \frac{4R}{r}|y|.
\]

Let $\gamma_{\varepsilon,u^\perp}(y)=(2\pi \varepsilon^2)^{-(n-1)/2}e^{-|y|^2/(2\varepsilon^2)},\; y\in u^{\perp}$. By the Fubini theorem, it follows that
\begin{align*}
    &\quad\enspace\int_{\mathbb{R}}f_\varepsilon(tu)\dd t=\int_{\mathbb{R}}\int_{\r}\mathbf{1}_K(tu+x)\gamma_\varepsilon(-x)\dd x\dd t
    \\
    &=\int_{\mathbb{R}}\int_{\mathbb{R}}\int_{u^{\perp}}\mathbf{1}_K(tu+su+y)\gamma_{\varepsilon, u^{\perp}}(y)\cdot \frac{1}{\sqrt{2\pi}\varepsilon}e^{-|s|^2/(2\varepsilon^2)}\dd y\dd s\dd t
  \\
    &=\int_{u^{\perp}}\int_{\mathbb{R}}\mathbf{1}_K(tu+y)\gamma_{\varepsilon, u^{\perp}}(y)\int_{\mathbb{R}} \frac{1}{\sqrt{2\pi}\varepsilon}e^{-|s|^2/(2\varepsilon^2)}\dd s\dd t\dd y
  =\int_{u^\perp}P_{K,\mathbb Ru}(y)\gamma_{\varepsilon, u^{\perp}}(y)\dd y.
\end{align*}
Hence, by $P_{K,\mathbb Ru}(0)\geq P_{K,\mathbb Ru}(y),\; y\in u^\perp$, we obtain
 \begin{align*}
     P_{K,\mathbb Ru}(0)-\int_{\mathbb{R}}f_\varepsilon(tu)\dd t=\int_{u^{\perp}}(P_{K,\mathbb Ru}(0)-P_{K,\mathbb Ru}(y))\gamma_{\varepsilon, u^{\perp}}(y)\dd y\geq0.
 \end{align*}
 
 Meanwhile, from that  $P_{K,\mathbb Ru}(0)-P_{K,\mathbb Ru}(y)
  \leq
  \frac{4R}{r}|y|
  $ for all $y\in u^{\perp}$ with $ |y|\le r/2$, $0\leq P_{K,\mathbb Ru}(y)\leq 2R$ for all $ y\in u^{\perp}$, and $\gamma_{1,u^{\perp}}\in\Sch(u^\perp) $, it follows that
\[
\begin{aligned}
0&\leq P_{K,\mathbb Ru}(0)
-\int_{\mathbb R}f_\varepsilon(tu)\dd t =
\int_{u^\perp}
(P_{K,\mathbb Ru}(0)-P_{K,\mathbb Ru}(y))
\gamma_{\varepsilon,u^\perp}(y)\dd y \\
&\leq
\frac{4R}{r}
\int_{u^\perp}|y|
\gamma_{\varepsilon,u^\perp}(y)\dd y
+
2R
\int_{\{y\in u^\perp:\,|y|>r/2\}}
\gamma_{\varepsilon,u^\perp}(y)\dd y \\
&=
\varepsilon\frac{4R}{r}
\int_{u^{\perp}}|y|\gamma_{1,u^\perp}(y)\dd y
+
2R
\int_{\{y\in u^{\perp}:|y|>r/(2\varepsilon)\}}
\gamma_{1,u^\perp}(y)\dd y
\to 0,\quad \text{as} \enspace \varepsilon\to 0^+.
\end{aligned}
\]
So,  $\sup\limits_{u\in\s}
  |P_{K,\mathbb Ru}(0)
-\int_{\mathbb R}f_\varepsilon(tu)\dd t|\to0$, as $\varepsilon\to 0^+$. 
Since $P_{K,\mathbb{R}u}(0)=2\rho_K(u)$, it follows that as $\varepsilon\to 0^+$,
  $\sup\limits_{u\in\s}
  |\frac12\int_{\mathbb{R}}f_\varepsilon(tu)\dd t
  -\rho_K(u)|\to0$,   and hence $ \sup\limits_{u\in\s}
  |h_\varepsilon(u)-\rho_K^{n-3}(u)|\to 0$.

\textbf{Step~2.} Prove that there exists a positive Borel measure
$\mu_K$ such that $\mathrm{R}_3^t\mu_K=\rho_K^{n-3}\dd\sigma$.

By the Riesz representation theorem (see \cite[Theorem 4.1.1]{AK}) and the fact that $G(n,3)$ is a compact Hausdorff
space, it is known that
$(C(G(n,3)))^*=\mathcal{M}(G(n,3))$.  

To prove the existence of $\mu_K$, it suffices to show that there exists a sequence  $\{\varepsilon_j\}_{j=1}^{\infty}$ such that $\varepsilon_j\to 0^+$ and $\mu_{\varepsilon_j}\stackrel{w^*}{\rightharpoonup}\mu_K$. If so, then  $\mu_K$ is positive since
$\int g\dd\mu_K=\lim_j\int g\dd\mu_{\varepsilon_j}\geq0$ for all nonnegative
continuous $g\in C(G(n,3))$; Moreover, for every $\varphi\in C(\s)$, 
\[
\begin{aligned}
  \int_{\s}\varphi\rho_K^{n-3}\dd\sigma
  &=\lim_{j\to\infty}\int_{\s}\varphi h_{\varepsilon_j}\dd\sigma=\lim_{j\to\infty}\int_{G(n,3)}
    (\mathrm{R}_3\varphi)\dd\mu_{\varepsilon_j}=\int_{G(n,3)}(\mathrm{R}_3\varphi)\dd\mu_K,
\end{aligned}
\]
which yields that $\mathrm{R}_3^t\mu_K=\rho_K^{n-3}\dd\sigma$.

Note that $C(G(n,3))$ is separable by \cite[Theorem 4.1.3]{AK} and the fact that $G(n,3)$ is a compact Hausdorff  metrizable space. Meanwhile, from $ \sup\limits_{u\in\s}
  |h_\varepsilon(u)-\rho_K^{n-3}(u)|\to 0$ as $\varepsilon\to 0^+$   and $\mathrm{R}_3^t\mu_\varepsilon
  =h_{\varepsilon}(u)\dd\sigma(u)$, there exist $\varepsilon_0>0$ and $M>0$ such that
\[
\mu_\varepsilon(G(n,3))=\int_{G(n,3)}\mathrm{R}_3(1)\dd\mu_\varepsilon
  =\int_{\s}h_\varepsilon\dd\sigma<M,\quad 0< \varepsilon<\varepsilon_0.
\]
By the Banach--Alaoglu theorem  (see \cite[Theorem 3.17]{RudinFA}),  a sequence $\{\varepsilon_j\}_{j=1}^{\infty}$ exists. 
\end{proof}

\noindent{\bf{Remark 2.}} (i). If in addition the boundary $\partial K$ 
 is of class $C^{\infty}$, we can show that $w_\varepsilon(E)$ converges to $\det(-D^2 P_{K,E}(0))$ uniformly  in $E\in G(n,3)$ as $\varepsilon\to0^+$, and therefore obtain
\begin{align*}
     \mu_K=(4\pi)^{3-n}\det(-D^2 P_{K,E}(0))\dd E.
\end{align*}
Since its proof is tedious, we would not dwell on it here.

(ii). For $n=4$,
 Zhang \cite[Lemma 1]{Zhang} established the above formula for $\mu_K$ when the boundary $\partial K$ is of class $C^{2}$.

\begin{proof}[\bf Proof of Theorem \ref{1.1}\thinspace\rm:]
From Theorem \ref{positive-representation} together with the definition of $m$-intersection body,  it follows that each origin-symmetric convex body $K$ in $\r$, $n\geq 4$, is a $3$-intersection body. Added with  Theorem \ref{m-interesection}, we conclude the proof of the case $m=3$.

  Assume $m=2$ and fix $H\in G(n,3)$. For every
$E\in G(n,2)$ with $E\subset H$, it follows that $\V_2((K\cap H)\cap E)
  \leq\V_2((L\cap H)\cap E).$
By the affirmative answer to the classical Busemann--Petty problem in $\mathbb{R}^3$
\cite{Gardner1}, it follows that $\V_3(K\cap H)\leq\V_3(L\cap H).$
Since this holds for every $H\in G(n,3)$,  by the affirmative answer to the generalized Busemann--Petty problem for $m=3$, we obtain that $\V_n(K)\leq\V_n(L)$.
\end{proof}
\vskip 5pt

\noindent\textbf{Declaration of competing interest}: We declare that we have no
conflict of interest.

\noindent\textbf{Data availability}: Not applicable.

\end{CJK*}

\end{document}